\documentclass[11pt,letterpaper,reqno]{amsart}
\usepackage{geometry}               
\usepackage{graphicx}
\usepackage{amssymb}
\usepackage{amsmath}
\usepackage{mathabx}
\usepackage{verbatim}
\usepackage{mathtools}
\usepackage{pdfsync}
\usepackage{hyperref}
\usepackage{calligra}
\usepackage{epstopdf}
\def\R{\mathbb{R}}
\def\T{\mathbb{T}}
\def\N{\mathbb{N}}
\def\Z{\mathbb{Z}}
\def\C{\mathbb{C}}

\def\D{\mathcal{D}}

\def\ov{\overline}

\renewcommand{\d}{\,{\rm d}}

\newtheorem{theorem}{Theorem}[section]
\newtheorem{corollary}[theorem]{Corollary}

\newtheorem{proposition}[theorem]{Proposition}
\newtheorem{lemma}[theorem]{Lemma}

\numberwithin{equation}{section}

\title[A discrete variational nonlinear Hausdorff--Young inequality]{A variational nonlinear Hausdorff--Young inequality in the discrete setting}
\author{Diogo Oliveira e Silva}
\address{
        Diogo Oliveira e Silva\\
        Hausdorff Center for Mathematics\\
        53115 Bonn, Germany}
\email{dosilva@math.uni-bonn.de}
\thanks{The author was partially supported by the Hausdorff Center for Mathematics and DFG grant CRC 1060.}

\date{\today}                                           

\begin{document}

\begin{abstract}
{\noindent Following the works of  Lyons \cite{Ly, Ly2} and Oberlin, Seeger, Tao, Thiele and Wright \cite{OSTTW}, we relate the variation of certain discrete curves on the Lie group $\text{SU}(1,1)$  to the corresponding variation of their linearized versions on the Lie algebra. 
Combining this with a discrete variational Menshov--Paley--Zygmund theorem, we establish a variational Hausdorff--Young inequality for a discrete version of the nonlinear Fourier transform \mbox{on $\text{SU}(1,1)$.} 
}
\end{abstract}


\maketitle
\section{Introduction}

The authors of \cite{OSTTW} considered a continuous version of the $\text{SU}(1,1)$-valued nonlinear Fourier transform, and established a variational Hausdorff--Young inequality which served as inspiration for much of the present work. In this paper, we investigate the corresponding discrete problem and establish Theorem \ref{VNLHY} below. Before stating it precisely, we introduce some context and background material.

\subsection{The generalized special unitary group and its Lie algebra} 

\noindent 
Following Tao and Thiele \cite{TT}, we are interested in a particular nonlinear Fourier transform taking values in the generalized special unitary group
$$\text{SU}(1,1)=\Big\{
\left(
\begin{array}{cc}
a & b \\
\ov{b} & \ov{a}
\end{array}\right)
: a,b\in\C
\text{ and }
 |a|^2-|b|^2=1
\Big\}.$$
This is a connected three-dimensional Lie group which is isomorphic to $\text{SL}(2,\R)$, and therefore neither compact nor simply connected. The corresponding Lie algebra is \mbox{given by}
$$\mathfrak{su}(1,1)=\Big\{
\left(
\begin{array}{cc}
ir & s-it \\
s+it  & -ir
\end{array}\right)
: r,s,t\in\R\Big\}.$$
Sometimes we shall identify a matrix of the form $\left(\begin{smallmatrix}a&b\\\ov{b} &\ov{a} \end{smallmatrix}\right)$ with its first row, thereby writing expressions like $(a,b)\in \text{SU}(1,1)$. The operator norm of a matrix $(a,b)\in\text{SU}(1,1)$ acting on the Hilbert space $\C^2$
is  given by the simple expression
\begin{equation}\label{opnorm}
\|(a,b)\|_{\text{op}}=|a|+|b|,
\end{equation}
see \cite[Lemma 33]{TT}.
 We endow $ \text{SU}(1,1)$ with the distance
\begin{equation}\label{distance}
d(X,Y)=\log(1+\|X^{-1}\cdot Y-I\|_{\text{op}}),
\end{equation}
where $I$ denotes the $2\times 2$ identity matrix. 
We also turn $\mathfrak{su}(1,1)$ into a normed algebra by endowing it with the operator norm.
For a further discussion of the Lie group $\text{SU}(1,1)$ and its Lie algebra, we refer the reader to \cite[\S 10.4]{S} and \cite[Appendix 2]{TT}.

\subsection{The nonlinear Fourier transform}
Let $\D$ denote the open unit disc in the complex plane, whose boundary coincides with the unit circle $\T$. Given a square-integrable, $\D$-valued sequence $F\in\ell^2(\Z,\D)$, its linear (inverse) Fourier transform is given by 
$$\widecheck{F}(z)=\lim_{N\to\infty}\sum_{|n|\leq N}F_n z^n, \;\;\;\;\;\;(z\in\T)$$
where the limit is taken in the sense of $\textup{L}^2(\T)$.
With the goal of defining a nonlinear analogue of the Fourier transform,
consider a compactly supported, $\D$-valued sequence $F$, 
which will often be referred to as a {\it potential}.
The associated transfer matrices $\{T_n\}$ are defined, for each $n\in\Z$ and $z\in\T$, by
$$T_n(z)=\left(
\begin{array}{cc}
1 & F_n z^n \\
\overline{F_n} z^{-n}  & 1
\end{array}\right)(1-|F_n|^2)^{-\frac12}.
$$
Note that $T_n(z)\in \text{SU}(1,1)$.
We define the nonlinear Fourier transform of the potential $F$ as an $\text{SU}(1,1)$-valued function on the unit circle $\T$, given by the expression
\begin{equation}\label{defNFT}
(a,b)(z)=\lim_{N\to\infty}\prod_{n=-N}^N 
T_n(z),\;\;\;(z\in\T)
\end{equation}
where the ordered product is actually finite.
In this case, a contour integration establishes the following nonlinear analogue of Plancherel's $\textup{L}^2-\textup{L}^2$ estimate:\footnote{Here and in what follows, $\int_\T$ denotes the integral with respect to Lebesgue measure on $\T$, normalized to have total mass 1.}
\begin{equation}\label{Plancherel}
\int_\T \log|a(z)|=-\frac{1}{2}\sum_{n\in\Z}\log(1-|F_n|^2),
\end{equation}
which can then be used to extend the definition of $(a,b)$ to $\D$-valued sequences $F$ which are merely assumed to be square-integrable.
If the sequence $F$ is further required to belong to $\ell^1$, then a clever but elementary convexity argument establishes the estimate
\begin{equation}\label{riemannlebesgue}
\sup_{z\in\T}\sqrt{\log|a(z)|}\leq\sum_{n\in\Z}\sqrt{\log((1-|F_n|^2)^{-\frac12})},
\end{equation}
which should be thought of as a nonlinear equivalent of the trivial $\textup{L}^1-\textup{L}^\infty$ estimate for the linear Fourier transform.
One would like to interpolate estimates \eqref{Plancherel} and \eqref{riemannlebesgue} to conclude a nonlinear version of the $\textup{L}^p-\textup{L}^{p'}$ Hausdorff--Young inequality when $1<p<2$. 
Here and throughout the paper, $p'=\frac p{p-1}$ denotes the exponent conjugate to $p$.
Even though standard interpolation tools are not available in the present nonlinear context, the following estimate  is a consequence of the seminal work of Christ and Kiselev \cite{CK, CK2} on the spectral theory of one-dimensional Schr\"{o}dinger operators: Given $1< p< 2$, there exists a constant $C_p<\infty$ such that, for every sequence $F\in\ell^p(\Z,\D)$, 
\begin{equation}\label{phausdorffyoung}
\Big\|\sqrt{\log |a(z)|}\Big\|_{\textup{L}^{p'}(\T)}\leq C_p\Big\| \sqrt{|\log(1-|F_n|^2)|}\Big\|_{\ell^p(\Z)}.
\end{equation}
The cases $p=1$ and $p=2$ of inequality \eqref{phausdorffyoung} boil down to \eqref{riemannlebesgue} and \eqref{Plancherel}, respectively.
For precise statements and proofs of these estimates, 
see \cite{TT} and the references therein. An interesting open problem is whether the constant $C_p$ in inequality \eqref{phausdorffyoung} can be chosen uniformly in $p$ as $p\rightarrow 2^-$. The answer is known to be affirmative in a variant for a particular Cantor group model of the continuous nonlinear Fourier transform \cite{K}, see also \cite{KOSR} for a recent related result.

\subsection{Main result} \noindent In this paper, we aim at a variational refinement of inequality \eqref{phausdorffyoung}. With this purpose in mind, consider the following truncated version of the linear and the nonlinear Fourier transforms of the potential $F$, respectively denoted by $\sigma(F)$ and $\gamma[F]$, given at height \mbox{$N\in\Z$ by}
\begin{equation}\label{defsigmagamma}
\sigma(F)(N;z)=\sum_{n=-\infty}^N\left(
\begin{array}{cc}
0 & F_n z^n \\
\overline{F_n} z^{-n}  & 0
\end{array}\right),\;\text{ and }\;
\gamma[F](N;z)=\prod_{n=-\infty}^N 
T_n(z).
\end{equation}
For each fixed $z\in\T$, it will be convenient to think of the maps $N\mapsto \gamma[F](N;z)$ and $N\mapsto \sigma(F)(N;z)$ as discrete curves taking values on $\text{SU}(1,1)$ and $\mathfrak{su}(1,1)$, respectively. To make the notation less cumbersome, we shall write  $\gamma_N(z)$ instead of $\gamma[F](N;z)$ whenever there is no danger of confusion, and similarly for the curve $\sigma(F)$. 
Given an exponent $1\leq r\leq\infty$, we will be interested in measuring the $r$-variation in the variable $N$ of the curves $\sigma=\sigma(F)$ and $\gamma=\gamma[F]$. If $1\leq r<\infty$, then these variations at a given point $z\in\T$ are defined as
\begin{align*}
\mathcal{V}_r(\sigma)(z)&=\sup_K\sup_{N_0<\ldots<N_K}\Big(\sum_{j=0}^{K-1} 
\|\sigma_{N_{j+1}}(z)-\sigma_{N_j}(z)\|_{\text{op}}^r\Big)^{\frac1r},\\
\mathcal{V}_r(\gamma)(z)&=\sup_K\sup_{N_0<\ldots<N_K}\Big(\sum_{j=0}^{K-1} d(\gamma_{N_j}(z),\gamma_{N_{j+1}}(z))^r\Big)^{\frac1r},
\end{align*}
where the suprema are taken over all strictly increasing finite sequences of integers $N_0<N_1<\ldots<N_K$ and over all integers $K$.
We can extend this to the case $r=\infty$ in the usual manner.
 The main result of this paper is the following theorem, 
 which should be compared to the estimates in \cite[p. 461]{OSTTW}. 

\begin{theorem}\label{VNLHY}
Let
$1\leq p<2$ and $r>p$. Then there exists a constant $C_{p,r}<\infty$ such that 
\begin{align}
\| \mathcal{V}_r(\gamma[F])\|_{\textup{L}^{p'}(\mathbb{S})}
+\|\mathcal{V}_r(\gamma[F])\|^{\frac1r}_{\textup{L}^{\frac{p'}r}(\T\setminus\mathbb{S})}
&\leq C_{p,r} \Big\|\log\Big(\frac{1+|F_n|}{1-|F_n|}\Big)\Big\|_{\ell^p(\Z)},\label{mainineq1}
\end{align}
for every $F\in\ell^p(\Z,\D)$.
Here $\mathbb{S}$ denotes the subset of the unit circle defined as follows:
\begin{equation}\label{defS}
\mathbb{S}=\mathbb{S}_{p,r}[F]
=\{z\in\T: \mathcal{V}_s(\gamma[F])(z)\leq 1\},
\end{equation}
where $s=r$ if $p<r<2$, and $s=\frac{p+2}2$ if $r\geq 2$.
\end{theorem}

\noindent 
 A few remarks may help to further orient the reader. 
\begin{itemize}
\item The usual logarithm (that appears for instance on the left-hand side of inequality \eqref{phausdorffyoung}) is hidden in the metric $d$ which we have placed on the group $\text{SU}(1,1)$.
\item The condition $r>p$ is sharp, except at the endpoint $p=1$. 
In this case, if $p=r=1$, then inequality \eqref{mainineq1} holds with $C_{1,1}\leq 1$. 
In the general case, \eqref{mainineq1} can only hold if $r\geq p$, as can easily be seen by considering potentials given by $F_n=\frac 12$ if $|n|\leq N$, and $F_n=0$ otherwise.  
Furthermore, if $p>1$, then \eqref{mainineq1} can only hold under the strict inequality $r>p$, as discussed after the statement of the corresponding linearized result, Proposition \ref{varMPZ} below.
\item  
 Extending Theorem \ref{VNLHY} to $p=2$, even when $r=\infty$, is a challenging open problem that would imply a  nonlinear analogue of Carleson's theorem on almost everywhere convergence of partial Fourier series, see \cite{MTT, MTT2} for an extended discussion.
More modestly, one can ask for a restricted weak type result at the endpoint $r=p$ when $1<p<2$.
\end{itemize}

\noindent Let us briefly comment on the proof of Theorem \ref{VNLHY}. 
It comprises two parts which constitute the upcoming sections. 
In \S \ref{sec:Variations}, we follow the adaptation of the work of Lyons \cite{Ly, Ly2} by Oberlin et al. \cite{OSTTW} to study variation norms on $\text{SU}(1,1)$. In particular, given a potential $F$, we show that the $r$-variation of the discrete curve $\gamma[F]$ can be controlled by the 
$r$-variation of the linearized curve $\sigma(F)$ plus an extra term that accounts for the possible presence of large jumps, as long as $r<2$.  
Most of the analysis is  local, and uses a modified induction on scale (or ``Bellman function'') argument inspired by \cite[Lemma C.3]{OSTTW} to control multiple steps of the curve $\gamma[F]$ at once. 
The corresponding step from \cite{OSTTW} relies on a key subdivision result, \cite[Lemma C.2]{OSTTW}, which asserts that any continuous curve can be decomposed into two proper sub-curves, each of which satisfies good $r$-variational  bounds.  
No such decomposition is available in the discrete setting. 
Instead, we decompose the partial sum process into a part encoding large range displacement and a part which tracks variation on a prescribed local scale, and analyze them separately.
This strategy already appeared in the works \cite{LL, LyY}.
One is then left with establishing a discrete variational version of the classical Menshov--Paley--Zygmund theorem. This step requires $r>p$, and is accomplished in \S \ref{sec:VMPZ} via an adaptation of the original argument of Christ--Kiselev \cite{CK2} to the variational setting.
The proof of Theorem \ref{VNLHY} is then completed in a few strokes in \S\ref{sec:corproofs}. 
\\

\noindent {\bf Notation.}
If $x,y$ are real numbers, we write $x=O(y)$ or $x\lesssim y$ if there exists a finite absolute constant $C$ such that $|x|\leq C|y|$. 
If we want to make explicit the dependence of the constant $C$  on some parameter $\alpha$, we  write $x=O_\alpha(y)$ or $x\lesssim_\alpha y$. 
We also write $x\vee y=\max\{x,y\}$ and $x\wedge y=\min\{x,y\}$.
The indicator function of a set $E$ is \mbox{denoted by ${\bf 1}_E$.}

\section{Variational estimates}\label{sec:Variations}

This section is devoted to controlling the variation of a curve on the Lie group $\text{SU}(1,1)$ in terms of the variation  of its linearized version. 
We refer the reader to \cite{DMT, DMT2, JSW, OSTTW, W} for background on variation seminorms, and proceed to explore some elementary properties which will be useful in our analysis. 
We shall state them for the exponentiated curve $\gamma=\gamma[F]$ only, but in each case it will be clear what the corresponding property for the linearized curve $\sigma=\sigma[F]$ should be. 
To make the forthcoming notation less cumbersome, we will occasionally drop the dependence on $F$ and $z$.
An application of Minkowski's inequality reveals that the $r$-variation is decreasing in the exponent $r$:
\begin{equation}\label{nestedVr}
\mathcal{V}_{s}(\gamma)\leq \mathcal{V}_{r}(\gamma),\textrm{ whenever }1\leq r\leq s\leq \infty.
\end{equation}
Implicit in the notation $\mathcal{V}_r(\gamma)$ is the fact that the variation is taken over the whole domain of definition of $\gamma$.  We will sometimes need to restrict the domain of the discrete curves under consideration. 
Given integers $a\leq b$, the restriction of the curve $\gamma$ to an interval $[a,b]\subset\Z$ will be denoted by $\gamma|_{[a,b]}$,
and the $r$-variation of the curve $\gamma$ on the interval $[a,b]$ is defined as
$$\mathcal{V}_r(\gamma;[a,b])=\sup_K\sup_{a=N_0<\ldots<N_K=b}\Big(\sum_{j=0}^{K-1} d(\gamma_{N_j},\gamma_{N_{j+1}})^r\Big)^{\frac1r}.$$
Two discrete curves $\gamma_1:[a,c]\to\text{SU}(1,1)$ and  $\gamma_2:[c,b]\to\text{SU}(1,1)$ can be concatenated provided $\gamma_1(c)=\gamma_2(c)$.
If $1\leq r\leq\infty$ and $\gamma_1\oplus\gamma_2:[a,b]\to\text{SU}(1,1)$ denotes the concatenation of $\gamma_1$ and $\gamma_2$, then it is a straightforward matter to establish the  triangle inequalities
\begin{equation}\label{triangle}
(\mathcal{V}_r^r(\gamma_1)+\mathcal{V}_r^r(\gamma_2))^{\frac1r}
\leq
\mathcal{V}_r(\gamma_1\oplus\gamma_2)
\leq 
\mathcal{V}_r(\gamma_1)+\mathcal{V}_r(\gamma_2).
\end{equation} 
If  $r=1$, then these inequalities combine into an equality. 
Moreover, both inequalities in \eqref{triangle} extend to an arbitrary (even countably infinite) number of summands.
 We seek to establish the following quantitative result.

\begin{proposition}\label{comparingvars}
Let $1\leq r<2$. Then there exists a constant $C_r<\infty$ such that
\begin{equation}\label{gammasigmaVr}
\mathcal{V}_r(\gamma)(z)\leq 
\mathcal{V}_r(\sigma)(z)+C_r\Big(\mathcal{V}_r^2(\sigma)(z)\wedge\mathcal{V}_r^r(\sigma)(z)
+\|F\|^{r-1}_{\ell^r(\Z)}  \Big\|\log\Big(\frac{1+|F_n|}{1-|F_n|}\Big)\Big\|_{\ell^r(\Z)}\Big),
\end{equation}
 for every potential $F$
and every $z\in\T$.
\end{proposition}

\begin{proof}[Proof of Proposition \ref{comparingvars}]
We lose no generality in assuming that $r>1$. 
Indeed, the 1-variation $\mathcal{V}_1$ is additive as quantified by \eqref{triangle}.
Moreover, identity \eqref{opnorm} and some straightforward algebra imply that the operator norm of the matrix $T_n(z)$ satisfies
\begin{equation}\label{opnormtransfermatrix}
\log(1+\|T_n(z)-I\|_{\text{op}})
=\frac{1}{2}\log\Big(\frac{1+|F_n|}{1-|F_n|}\Big).
\end{equation}
These two observations imply
$$\mathcal{V}_1(\gamma)(z)
=\sum_{n\in\Z} \mathcal{V}_1(\gamma;[n-1,n])(z)
=\frac 12\sum_{n\in\Z}\log\Big(\frac{1+|F_n|}{1-|F_n|}\Big), \text{ for every } z\in\T,$$
which in turn implies a stronger form of \eqref{gammasigmaVr}.
For $1<r<2$, the somewhat lengthy proof of Proposition \ref{comparingvars} is divided into three steps. 
The first two steps deal with the local analysis.\\

\noindent{\it Step 1.} {\it For every $1< r<2$, there exist constants $0<\delta<\frac 12$ and $K<\infty$ with the following property. For every potential $F$, the curves $\gamma=\gamma[F]$ and $\sigma=\sigma(F)$ satisfy
\begin{equation}\label{step2}
\|\log(\gamma_M^{-1}(z)\cdot\gamma_N(z))-(\sigma_N(z)-\sigma_M(z))\|_{\emph{\text{op}}}
\leq K \mathcal{V}_r^2(\sigma;[M,N])(z)
\end{equation}
 for every $z\in\T$ and  integers $M\leq N$, provided $\mathcal{V}_r(\sigma;[M,N])(z)\leq\delta$.}
\vspace{.2cm}

\noindent Start by noting that, if $\delta$ is sufficiently small, then the matrix $(\gamma_M^{-1}\cdot\gamma_{N})(z)=\prod_{M<n\leq N}T_{n}(z)$ is sufficiently close to the identity  so that the logarithm is well-defined.
If $N=M+1$, then the left-hand side of inequality \eqref{step2} can be determined exactly.
Indeed, a straightforward computation yields
$$\log(\gamma_{N-1}^{-1}\cdot\gamma_{N})
=g(F_{N}) (0 , F_{N} z^{N} ),$$
where the function $g$ is given by
$g(t)=(2|t|)^{-1}{\log\Big(\frac{1+|t|}{1-|t|}\Big)}.$
\noindent It follows that
\begin{equation}\label{unitstep}
\|\log(\gamma_{N-1}^{-1}\cdot\gamma_{N})-(\sigma_{N}-\sigma_{N-1})\|_{\text{op}}
=|g(F_{N})-1|\|(0 , F_{N} z^{N} )\|_{\text{op}}
=|g(F_{N})-1||F_{N}|.
\end{equation}
A Taylor expansion of the function $g$ to order 2 reveals that this term is
$O(|F_N|^3)$ if $|F_N|\leq\frac 12$. 
To handle the case $N>M+1$, we perform an induction on scale argument which follows the corresponding part of the proof of \cite[Lemma C.3]{OSTTW}, with some differences that we highlight below.
Let us first establish inequality \eqref{step2} in the case when the quantity 
 $\mathcal{V}_r(\sigma;[M,N])$
is sufficiently small, depending on
 the potential $F$. 
We have that
\begin{align}
\log (\gamma_M^{-1}\cdot\gamma_N)+\frac 12\Big(\sum_{M<n\leq N} \log(1-|F_n|^2) \Big)I
&=\log\Big(\prod_{M<n\leq N} \Big(I+(0  , F_n z^n)
\Big)\Big)\label{decomposelogs}\\
&=\log(I+(\sigma_N-\sigma_M)+H_M^N),\notag
\end{align}
where $H_M^N=H_M^N[F]$ denotes the higher order terms that appear when we expand the product in the argument of the logarithm. 
Taylor expanding the logarithm, we see that
\begin{equation}\label{Est0}
\|\log(I+(\sigma_N-\sigma_M)+H_M^N)-(\sigma_N-\sigma_M)\|_{\text{op}}
=O(\|\sigma_N-\sigma_M\|_{\text{op}}^2),
\end{equation}
provided
$\|\sigma_N-\sigma_M\|_{\text{op}}$ is sufficiently small.
On the other hand,
\begin{equation}\label{Est1}
\|\sigma_N-\sigma_M\|_{\text{op}}
\leq
\mathcal{V}_r(\sigma;[M,N]).
 \end{equation}
Further notice that, if $|F_n|\leq\frac12$ for every $n$, then
\begin{equation}\label{Est2}
\frac 12\sum_{M<n\leq N} \big|\log(1-|F_n|^2)\big|
\leq\sum_{M<n\leq N} |F_n|^2
\leq\Big(\sum_{M<n\leq N} |F_n|^r\Big)^{\frac 2r}
\leq \mathcal{V}^2_r(\sigma;[M,N]).
\end{equation}
Recalling \eqref{decomposelogs}, estimates \eqref{Est0}, \eqref{Est1} and \eqref{Est2} imply that inequality \eqref{step2} holds provided the quantity $\mathcal{V}_r(\sigma;[M,N])$ is sufficiently small, depending on 
the potential $F$.
This provides the base of the induction scheme, which we now perform in order to remove the undesirable dependence on the potential.
To specify the induction step, fix a constant of proportionality $\theta$ with $\frac12<\theta<\frac 34$.\\

\noindent {\bf Inductive Claim.} {\it If estimate \eqref{step2} holds when $\mathcal{V}_r(\sigma;[M,N])<\varepsilon$, for some $0<\varepsilon<\delta$, then it also holds when $\mathcal{V}_r(\sigma;[M,N])<\theta^{-\frac1r} \varepsilon$, provided $K$ is chosen sufficiently large 
and $\delta$ is sufficiently small.}
\vspace{.2cm}

\noindent We emphasize that both parameters $K,\delta$ are allowed to depend only on the exponent $r$, but that $\delta$ will be chosen as a function of $K$ in the course of the proof.
Let us proceed to prove the Inductive Claim. 
Consider a potential $F$ and the corresponding curve $\sigma=\sigma(F)$, such that $A:=\mathcal{V}_r(\sigma;[M,N])$ satisfies $A<\theta^{-\frac1r}\varepsilon$, for some $0<\varepsilon<\delta$. In this case, we have in particular that $A<2\delta$.
By a {\it good subdivision} of the curve $\sigma|_{[M,N]}$ we mean a decomposition $\sigma|_{[M,N]}=\sigma_1\oplus\sigma_2$ into discrete curves $\sigma_1=\sigma|_{[M,L]}$ and $\sigma_2=\sigma|_{[L,N]}$, such that the corresponding variations satisfy
\begin{equation}\label{max}
\mathcal{V}_r(\sigma_1)\vee \mathcal{V}_r(\sigma_2)=\alpha^{\frac1r}A,
\end{equation}
for some $1-\theta\leq \alpha\leq \theta$.
We split the analysis into two  cases, according to whether or not a good subdivision of $\sigma|_{[M,N]}$ exists.

\noindent Suppose first that there exists a good subdivision $\sigma|_{[M,N]}=\sigma_1\oplus\sigma_2$. 
In this case, the argument is analogous to the one in \cite[Lemma C.3]{OSTTW}, but we repeat it here for the convenience of the reader.
We can assume that the maximum in \eqref{max} is attained by the first piece, $\mathcal{V}_r(\sigma_1)=\alpha^{\frac1r}A$, from which it follows by the reverse triangle inequality that 
$$\mathcal{V}_r(\sigma_2)\leq (1-\alpha)^{\frac1r}A.$$
Since $\alpha\leq\theta$, we have that  $\mathcal{V}_r(\sigma_1)<\varepsilon$,  and therefore $\mathcal{V}_r(\sigma_2)<\varepsilon$ as well. By the induction hypothesis, it follows that
\begin{align}
&\|\log(\gamma_M^{-1}\cdot\gamma_L)-(\sigma_L-\sigma_M)\|_{\text{op}}\leq K(\alpha^{\frac1r}A)^2,\text{ and }\label{EstMK}\\
&\|\log(\gamma_L^{-1}\cdot\gamma_N)-(\sigma_N-\sigma_L)\|_{\text{op}}\leq K((1-\alpha)^{\frac1r}A)^2.\label{EstKN}
\end{align}
\noindent Estimate \eqref{EstMK} and $A<2\delta$ together imply
\begin{align*}
\|\log(\gamma_M^{-1}\cdot\gamma_L)\|_{\text{op}}
&\leq \|\sigma_L-\sigma_M\|_{\text{op}}+K(\alpha^{\frac1r}A)^2
\leq \mathcal{V}_r(\sigma;[M,N])+KA^2\\
&=A(1+KA)\leq A(1+2K\delta)=O(A),
\end{align*}
provided $\delta$ is chosen sufficiently small depending on $K$. 
Estimate \eqref{EstKN} similarly implies
$$\|\log(\gamma_L^{-1}\cdot\gamma_N)\|_{\text{op}}=O(A).$$
Invoking the simple estimates
\begin{equation}\label{ExpLogEst}
\|\exp(X)-I-X\|_{\text{op}}=O(\|X\|_{\text{op}}^2)
\text{ and }
\|\log(I+X)-X\|_{\text{op}}=O(\|X\|_{\text{op}}^2),
\end{equation}
both of which hold with implicit constants less than $2$ as long as the matrix $X$ satisfies $\|X\|_{\text{op}}\leq\frac 12$, we then have that
\begin{equation}\label{BCK}
\|\log(\gamma_M^{-1}\cdot\gamma_N)-\log(\gamma_M^{-1}\cdot\gamma_L)-\log(\gamma_L^{-1}\cdot\gamma_N)\|_{\text{op}}=O(A^2),
\end{equation}

\noindent if $\delta$ is sufficiently small.
Using the triangle inequality and estimates \eqref{EstMK}, \eqref{EstKN} and \eqref{BCK}, it follows that
\begin{equation}\label{LogSigmaKO}
\|\log(\gamma_M^{-1}\cdot\gamma_N)-(\sigma_N-\sigma_M)\|_{\text{op}}
\leq 
K(\alpha^{\frac2r}+(1-\alpha)^{\frac2r})A^2+O(A^2).
\end{equation}

\noindent Since $1< r<2$, we have that the quantity $\alpha^{\frac2r}+(1-\alpha)^{\frac2r}<1$ is bounded away from 1, for every $1-\theta\leq\alpha\leq\theta$. 
If $K$ is large enough, depending only on $r$, we can bound the right-hand side of \eqref{LogSigmaKO} by $KA^2$, and thus have established \eqref{step2} if $A<\theta^{-\frac1r}\varepsilon$. This finishes the analysis of the case when a good subdivision exists.

\noindent Let us now assume that no good subdivision of the curve $\sigma|_{[M,N]}$ exists. 
In this case, there exists a one step sub-curve of $\sigma|_{[M,N]}$, say $\sigma|_{[L-1,L]}$,  which accounts for most of its variation. 
Quantitatively,
\begin{equation}\label{quant}
\theta^{\frac1r}-(1-\theta)^{\frac1r}<\frac{\mathcal{V}_r(\sigma;[L-1,L])}{\mathcal{V}_r(\sigma;[M,N])}=\frac{|F_L|}{A}\leq 1, 
\end{equation}
\begin{equation}\label{quant2}
 \mathcal{V}_r(\sigma;[M,L-1])<(1-\theta)^{\frac1r}A,\text{ and }
{\mathcal{V}_r(\sigma;[L,N])}<(1-\theta)^{\frac1r}A.
\end{equation}
Recalling \eqref{unitstep}, we have that
\begin{equation}\label{NewEst1}
\|\log(\gamma_{L-1}^{-1}\cdot\gamma_L)-(\sigma_L-\sigma_{L-1})\|_{\text{op}}
=|g(F_L)-1||F_L|,
\end{equation}
and therefore
\begin{equation}\label{firstlog}
\|\log(\gamma_{L-1}^{-1}\cdot\gamma_L)\|_{\text{op}}\leq |F_L|+O(|F_L|^3).
\end{equation}
The first inequality in \eqref{quant2} implies $\mathcal{V}_r(\sigma;[M,L-1])<\varepsilon$, and so by induction hypothesis 
\begin{equation}\label{NewEst2}
\|\log(\gamma_M^{-1}\cdot\gamma_{L-1})-(\sigma_{L-1}-\sigma_M)\|_{\text{op}}
\leq
K((1-\theta)^{\frac1r}A)^2.
\end{equation}
Since $A<2\delta$, it follows that
\begin{align}
\|\log(\gamma_M^{-1}\cdot\gamma_{L-1})\|_{\text{op}}
&\leq \mathcal{V}_r(\sigma;[M,L-1])+K(1-\theta)^{\frac2r}A^2\notag\\
&\leq (1-\theta)^{\frac1r}A(1+2K(1-\theta)^{\frac1r}\delta)
\leq 2(1-\theta)^{\frac1r}A
\leq \frac{2(1-\theta)^{\frac1r}}{\theta^{\frac1r}-(1-\theta)^{\frac1r}}|F_L|,\label{secondlog}
\end{align}
provided $\delta$ is chosen sufficiently small as a function of $K$. 
An identical analysis applies to the sub-curve $\sigma|_{[L,N]}$ and yields
\begin{equation}\label{NewEst3}
\|\log(\gamma_L^{-1}\cdot\gamma_N)-(\sigma_N-\sigma_L)\|_{\text{op}}
\leq
K((1-\theta)^{\frac1r}A)^2,
\end{equation}
\begin{equation}\label{thirdlog}
\|\log(\gamma_L^{-1}\cdot\gamma_N)\|_{\text{op}}
\leq
\frac{2(1-\theta)^{\frac1r}}{\theta^{\frac1r}-(1-\theta)^{\frac1r}}|F_L|.
\end{equation}
From estimates \eqref{firstlog}, \eqref{secondlog} and \eqref{thirdlog} it follows, as before, that
\begin{equation}\label{NewEst4}
\|\log(\gamma_M^{-1}\cdot\gamma_N)-\log(\gamma_M^{-1}\cdot\gamma_{L-1})-\log(\gamma_{L-1}^{-1}\cdot\gamma_L)-\log(\gamma_L^{-1}\cdot\gamma_N)\|_{\text{op}}=O(|F_L|^2)=O(A^2).
\end{equation}
By the triangle inequality and estimates \eqref{NewEst1}, \eqref{NewEst2}, \eqref{NewEst3} and \eqref{NewEst4}, we then have
$$\|\log(\gamma_M^{-1}\cdot\gamma_N)-(\sigma_N-\sigma_M)\|_{\text{op}}
\leq 
|g(F_L)-1||F_L|+2K(1-\theta)^{\frac2r}A^2+O(A^2).
$$
Since $2(1-\theta)^{\frac2r}<1$, the right-hand side of this inequality can again be majorized by $KA^2$, provided $K$ is sufficiently large (depending on $r$). This concludes the analysis when no good subdivision exists. The Inductive Claim is now  proved, and this settles Step 1.\\

\noindent{\it Step 2.} {\it Given $1< r<2$, let $\delta$ be the constant promised by Step 1. 
Then, for every potential $F$, the curves $\gamma=\gamma[F]$  and $\sigma=\sigma(F)$ satisfy
\begin{equation}\label{step3}
\mathcal{V}_r(\gamma)(z)= \mathcal{V}_r(\sigma)(z) +O_r(\mathcal{V}^2_r(\sigma)(z)),
\end{equation}
for every $z\in\T$, provided $\mathcal{V}_r(\sigma)(z) \leq\delta$.}
\vspace{.2cm}

\noindent The elementary estimates in \eqref{ExpLogEst} imply that short distances in the Lie group and the Lie algebra agree modulo a quadratically small error: Given a matrix $X\in\mathfrak{su}(1,1)$ of sufficiently small operator norm, 
$$d(I,\exp(X))=\log(1+\|\exp(X)-I\|_{\text{op}})=\|X\|_{\text{op}}+O(\|X\|^2_{\text{op}}).$$
Consequently, if $\delta$ is small enough, then, for any integers $M\leq N$,
$$d(\gamma_M,\gamma_N)
=d(I,\exp(\log(\gamma^{-1}_M\cdot\gamma_N)))
=\|\log(\gamma^{-1}_M\cdot\gamma_N)\|_{\text{op}}+O(\|\log(\gamma^{-1}_M\cdot\gamma_N)\|^2_{\text{op}}).$$
In view of Step 1, we then have
$$d(\gamma_M,\gamma_N)=\|\sigma_N-\sigma_M\|_{\text{op}}+O(\mathcal{V}_r^2(\sigma;[M,N])).$$
\noindent Given $\varepsilon>0$, choose an integer $R$ which is large enough, so that $\mathcal{V}_r(\gamma)\leq \mathcal{V}_r(\gamma;[-R,R])+\varepsilon$.
 For any partition $-R=N_0<N_1<\ldots<N_K=R$, we may crudely estimate  
 $$O(\mathcal{V}_r^2(\sigma;[N_j,N_{j+1}]))\leq \mathcal{V}_r(\sigma;[-R,R])O(\mathcal{V}_r(\sigma;[N_j,N_{j+1}])),$$ 
 and conclude, by taking the $\ell^r$ sum in the $j$ index, that
\begin{align*}
\Big(\sum_{j=0}^{K-1}d(\gamma_{N_{j+1}},\gamma_{N_j})^r\Big)^{\frac1r}
=\Big(\sum_{j=0}^{K-1} \|\sigma_{N_{j+1}}-\sigma_{N_{j}}\|_{\text{op}}^r\Big)^{\frac1r}+O(\mathcal{V}^2_r(\sigma;[-R,R])).
\end{align*}
Taking the suprema over all partitions and over $K$, and then letting $\varepsilon\to 0^+$, we obtain the result.
\noindent This concludes the proof of Step 2, and with it the local analysis is finished.  
The next and final step deals with the global part of the analysis.\\

\noindent{\it Step 3.} {\it Global analysis and conclusion of the proof.}
\vspace{.2cm}

\noindent Given $1<r<2$, let $\delta$ be the constant promised by Step 1.
Consider a potential $F$.  
The strategy is to decompose the curves $\gamma=\gamma[F]$ and $\sigma=\sigma(F)$ into finitely many small parts and large jumps, depending on $\delta$, each of which have some desirable properties. 
With this purpose in mind, consider the set of locations of the large jumps, $\mathbb{J}=\{n\in\Z: |F_n|>\frac{\delta}2\}$. The following bound for its cardinality, denoted $J=\# \mathbb{J}$, follows from Chebyshev's inequality:
\begin{equation}\label{Chebyshev}
J<\Big(\frac{2\|F\|_{\ell^r}}{\delta}\Big)^r.
\end{equation}
Let $n_1<n_2<\dots<n_J$ denote the elements of the set $\mathbb{J}$. For each $1\leq j\leq J$, set 
$\sigma_j^\ell=\sigma|_{[n_j-1,n_j]}$ 
and
$\gamma_j^\ell=\gamma|_{[n_j-1,n_j]}$. 
These induce natural decompositions 
\begin{align}
\gamma&=\gamma_1^s\oplus\gamma_1^\ell\oplus\gamma_2^s\oplus\gamma_2^\ell\oplus\ldots\oplus\gamma_{J}^s\oplus\gamma_{J}^\ell\oplus\gamma_{J+1}^s,\label{decomposegamma}\\
\sigma&=\sigma_1^s\oplus\sigma_1^\ell\oplus\sigma_2^s\oplus\sigma_2^\ell\oplus\ldots\oplus\sigma_{J}^s\oplus\sigma_{J}^\ell\oplus\sigma_{J+1}^s.\label{decomposesigma}
\end{align}
By the triangle inequality,
\begin{equation}\label{EstTri}
\mathcal{V}_r(\gamma)
\leq\sum_{j=1}^{J}\mathcal{V}_r(\gamma_j^\ell)+\sum_{j=1}^{J+1}\mathcal{V}_r(\gamma_j^s).
\end{equation}
Using  H\"older's inequality, estimate \eqref{Chebyshev}, and the fact that $\mathcal{V}_r(\gamma_j^\ell)=\frac 12\log(\frac{1+|F_{n_j}|}{1-|F_{n_j}|})$, we may estimate the contributions coming from the large jumps as follows:
\begin{equation}\label{EstLarge}
\sum_{j=1}^{J}\mathcal{V}_r(\gamma_j^\ell)
\leq J^{\frac1{r'}} \Big(\sum_{j=1}^{J}\mathcal{V}^r_r(\gamma_j^\ell)\Big)^{\frac1r}
\lesssim_r \|F\|_{\ell^r}^{r-1}\Big(\sum_{j=1}^J \mathcal{V}_r^r(\gamma_j^\ell)\Big)^{\frac 1r}
=\frac{\|F\|_{\ell^r}^{r-1}}2\Big(\sum_{n\in\mathbb{J}}\log\Big(\frac{1+|F_n|}{1-|F_n|}\Big)^r\Big)^{\frac 1r}.
\end{equation}
To estimate the contributions coming from the small jumps, consider the curve $\widetilde{\sigma}=\sigma(F{\bf 1}_{\Z\setminus\mathbb{J}})$ associated to the potential $F$ with the large jumps removed. 
We claim that the curve $\widetilde{\sigma}$ can be subdivided into $L=O(\delta^{-r} \mathcal{V}^r_r(\sigma))$ sub-curves $\{\widetilde{\sigma}_j\}_{j=1}^L$, in such a way that $\mathcal{V}_r(\widetilde{\sigma}_j)\leq \delta$, for every $1\leq j\leq L$. More precisely, each curve $\sigma_j^s$ admits
 a further decomposition 
$$\sigma_j^s=\bigoplus_{k=1}^{L_j}\sigma_{j,k}^s$$
with each $\mathcal{V}_r(\sigma_{j,k}^s)\leq \delta$, where $L:=\sum_{j=1}^{J+1} L_j=O(\delta^{-r} \mathcal{V}^r_r(\sigma))$. 
This can be seen via iterating a decomposition procedure similar to the one used in the proof of \cite[Lemma C.2]{OSTTW}.
The details are left to the reader.
By Step 2, the $r$-variation of the corresponding curves $\gamma_{j,k}^s$ satisfies
$$\mathcal{V}_r(\gamma_{j,k}^s)=\mathcal{V}_r(\sigma_{j,k}^s)+O(\mathcal{V}^2_r(\sigma_{j,k}^s)).$$
Applying the triangle inequality and H\"older's inequality as before, we estimate
\begin{align*}
\sum_{j=1}^{J+1}\mathcal{V}_r(\gamma_j^s)
&\leq
\sum_{j=1}^{J+1}\sum_{k=1}^{L_j}\mathcal{V}_r(\gamma_{j,k}^s)
=
\sum_{j=1}^{J+1}\sum_{k=1}^{L_j}\big(\mathcal{V}_r(\sigma_{j,k}^s)+O(\mathcal{V}^2_r(\sigma_{j,k}^s))\big)
\\
&\leq
L^{\frac{1}{r'}}\Big(\sum_{j=1}^{J+1}\sum_{k=1}^{L_j}\mathcal{V}^r_r(\sigma_{j,k}^s)\Big)^{\frac{1}{r}}
+CL^{\frac{1}{r'}}\Big(\sum_{j=1}^{J+1}\sum_{k=1}^{L_j}\mathcal{V}^{2r}_r(\sigma_{j,k}^s)\Big)^{\frac{1}{r}}.
\end{align*}
The reverse triangle inequality and the bound $\mathcal{V}_r(\sigma_{j,k}^s)\leq \delta$ then imply
$$\sum_{j=1}^{J+1}\mathcal{V}_r(\gamma_j^s)
\leq
L^{\frac 1{r'}}(1+C\delta) \mathcal{V}_r(\sigma).$$
Recalling that $L=O(\delta^{-r} \mathcal{V}^r_r(\sigma))$, and that $C,\delta$ depend only on $r$, we have that
\begin{equation}\label{gettingtoVrr}
\sum_{j=1}^{J+1}\mathcal{V}_r(\gamma_j^s)\lesssim_r \mathcal{V}_r^r(\sigma).
\end{equation}
Since $1<r<2$, we may then take an appropriate geometric mean to verify that
\begin{equation}\label{EstSmall}
\sum_{j=1}^{J+1}\mathcal{V}_r(\gamma_j^s)\leq \mathcal{V}_r(\sigma)+C_r (\mathcal{V}^2_r(\sigma)\wedge\mathcal{V}^r_r(\sigma)).
\end{equation}
Inequality \eqref{gammasigmaVr} follows from estimates \eqref{EstTri}, \eqref{EstLarge} and \eqref{EstSmall}, and this finishes the proof of the proposition.
\end{proof}

\noindent Proposition \ref{comparingvars} controls the $r$-variation of the curve $\gamma$ in terms of that of $\sigma$, plus an extra term that accounts for the large jumps.  
It is reasonable to ask whether, in the converse direction, the $r$-variation of $\sigma$ can be controlled by that of $\gamma$. 
An analysis of the previous proof reveals this to be the case, without the need for an extra term as before. We formulate this observation as our next result, which is going to play a role in extending the range of exponents for which Theorem \ref{VNLHY} holds.

\begin{corollary}\label{revcomparingvars}
Let $1\leq r<2$. Then there exists a constant $C_r<\infty$ such that
\begin{equation}\label{sigmagammaVr}
\mathcal{V}_r(\sigma)(z)\leq \mathcal{V}_r(\gamma)(z)+C_r(\mathcal{V}_r^2(\gamma)(z)\wedge\mathcal{V}_r^r(\gamma)(z)),
\end{equation}
 for every potential $F$
and every $z\in\T$.
\end{corollary}

\begin{proof}
As observed in the course of the proof of \cite[Lemma C.3]{OSTTW} via a continuity argument, the fact that estimate \eqref{step3} holds provided $\mathcal{V}_r(\sigma)(z)\leq\delta$, for sufficiently small $\delta$, implies
$$\mathcal{V}_r(\sigma)(z)= \mathcal{V}_r(\gamma)(z) +O_r(\mathcal{V}^2_r(\gamma)(z)),$$
provided $\mathcal{V}_r(\gamma)(z)\leq\delta$.
One is thus left with establishing the analogue of Step 3 of the proof of Proposition \ref{comparingvars}.
The estimate of the contributions from the small jumps is analogous and yields
\begin{equation}\label{gettingtoVrrGamma}
\sum_{j=1}^{J+1}\mathcal{V}_r(\sigma_j^s)\lesssim_r \mathcal{V}_r^r(\gamma).
\end{equation}
Since $\|F\|_{\ell^r} 
\leq \mathcal{V}_r(\gamma)$, we can estimate the contributions from the large jumps as follows:
\begin{equation}\label{SigmaEstLarge}
\sum_{j=1}^{J}\mathcal{V}_r(\sigma_j^\ell)
\lesssim_r \|F\|_{\ell^r}^{r-1}\Big(\sum_{j=1}^J \mathcal{V}_r^r(\sigma_j^\ell)\Big)^{\frac 1r}
={\|F\|_{\ell^r}^{r-1}}\Big(\sum_{n\in\mathbb{J}}|F_n|^r\Big)^{\frac 1r}
\leq \|F\|_{\ell^r}^r\leq \mathcal{V}_r^r(\gamma).
\end{equation}
The upper bounds given by \eqref{gettingtoVrrGamma} and \eqref{SigmaEstLarge} coincide, and the argument can be finished as before to yield \eqref{sigmagammaVr}.
\end{proof}

\section{A discrete variational Menshov--Paley--Zygmund theorem}\label{sec:VMPZ}

This section is devoted to controlling the variation of the linearized curve $\sigma(F)$ in terms of the potential $F$. The following result should be compared to the classical Menshov--Paley--Zygmund theorem on the torus.

\begin{proposition}\label{varMPZ}
Let $1\leq p<2$ and $r>p$. Then there exists a constant $C_{p,r}<\infty$ such that
\begin{equation}\label{vMPZineq}
\|\mathcal{V}_r(\sigma(F))\|_{\textup{L}^{p'}(\T)}\leq C_{p,r}\|F\|_{\ell^p(\Z)},
\end{equation}
for every $F\in\ell^p(\Z,\D)$.
\end{proposition}
\noindent The condition $r>p$ is sharp, except at the endpoint $p=1$. 
In this case, if $p=r=1$, then inequality \eqref{vMPZineq} turns into an equality with $C_{1,1}= 1$. 
In the general case, \eqref{vMPZineq} can only hold if $r\geq p$, as can easily  be seen by considering  truncated potentials given by $F_n=\frac 12$ if $|n|\leq M$, and $F_n=0$ otherwise. 
These of course coincide with (a multiple of) the Fourier transform of the Dirichlet kernel, defined for $M\geq 1$ to be
$${D}_M(x)=\sum_{n=-M}^{M} e^{inx}=\frac{\sin((M+\frac12)x)}{\sin \frac x2}.$$
Moreover, if $p>1$, then $r>p$ is a necessary condition for \eqref{vMPZineq} to hold.
This follows from an adaptation of the argument in 
\cite[\S 2]{OSTTW}
by testing \eqref{vMPZineq} against the Fourier transform of the de la Vall\'ee-Poussin kernel. If
${K}_M=(M+1)^{-1}\sum_{n=0}^M {D}_n$
denotes the Fej\'er kernel, then set ${V}_M=2{K}_{2M+1}-{K}_M$, and define the potential {$F^M\in\ell^p$ via}
\begin{displaymath}
2F^M_n
=\frac 1{2\pi}\int_0^{2\pi} {V}_M(x) e^{-inx}\d x
= \left\{ \begin{array}{ll}
1, & \textrm{if $|n|\leq M+1$},\\
\frac{2M+2-|n|}{M+1},& \textrm{if $M+1\leq|n|\leq 2M+2$},\\
0, & \textrm{if $|n|\geq 2M+2$}.\\
\end{array} \right.
\end{displaymath}
Since $\|F^M\|_{\ell^p}=O_p(M^{\frac1p})$, a computation analogous to the one in \cite[\S 2]{OSTTW} shows that, for sufficiently \mbox{large $M$,}
$$\frac{\|\mathcal{V}_p(\sigma(F^M))\|_{\textup{L}^{p'}}}{\|F^M\|_{\ell^{p}}}\geq c_p(\log M)^{\frac1{p'}},$$
hence the claimed necessity.
 For small values of $|F_n|$, the nonlinear Fourier transform is well approximated by the linear 
 Fourier transform, as can be seen by linearizing in $F$. It follows from this general principle and the discussion in the last paragraph that the condition $r>p$ is necessary for Theorem \ref{VNLHY} to hold, at least in the interesting case when $p>1$.

\noindent It seems possible to adapt the high-powered methods from \cite{OSTTW} to establish the case $p=2$ of Proposition \ref{varMPZ}, at which point the result would follow from interpolation with a trivial estimate at $p=1$. However, the purpose of this section is to provide an elementary proof when $p<2$ based on embedding $\ell^p(\Z)$ into $\textup{L}^p(\R)$.

\begin{proof}[Proof of Proposition \ref{varMPZ}]
For $x,y\in\R$, consider the kernel
\begin{displaymath}
K(x,y) = \left\{ 
\begin{array}{ll}
e^{ix[y]} & \textrm{if $x\in [0,2\pi]$,}\\
0 & \textrm{otherwise,}
\end{array} \right.
\end{displaymath}
where $[y]$ denotes the largest integer smaller than $y$. The associated integral operator is
$$Tf(x)=\int_\R K(x,y)f(y)\d y.$$
Since $1\leq p<2$, the operator $T$ is bounded from $\textup{L}^p(\R)$ to $\textup{L}^{p'}(\R)$, with operator norm satisfying $\|T\|_{p, p'}\leq (2\pi)^{\frac1{p'}}$. To verify this, consider the averages 
$\langle f\rangle_n=\int_n^{n+1}f(y)\d y$. 
Identifying $z=e^{ix}$, we have that
\begin{equation*}
\frac1{2\pi}\int_\R |Tf(x)|^{p'}\d x
=\int_\T\Big|\sum_{n\in\Z}\langle f\rangle_n z^{n}\Big|^{p'}
=\|{\{\langle f\rangle_n\}}\;\widecheck{}\;\|_{\textup{L}_z^{p'}(\T)}^{p'}.
\end{equation*}
By the (dual) Hausdorff--Young inequality on $\T$ and H\"{o}lder's inequality,
$$\frac1{2\pi}\int_\R |Tf(x)|^{p'}\d x\leq \|\{\langle f\rangle_n\}\|_{\ell^p(\Z)}^{p'}\leq \|f\|_{\textup{L}^p(\R)}^{p'},$$

\noindent and the claimed boundedness of $T$ follows. We proceed to define the truncated operators
$$T_{\leq}f(x,N)=\int_{-\infty}^NK(x,y)f(y)\d y.$$

\noindent Given a sequence $F\in\ell^p(\Z,\D)$, we construct a companion function $f\in \textup{L}^p(\R,\D)$ via 
$$f=\sum_{n\in\Z}F_n\mathbf{1}_{[n,n+1)}.$$
For every $0\leq x\leq 2\pi$ and $N\in\Z$, we have that
$$T_{\leq}f(x,N)=\sum_{n<N} F_n e^{inx}.$$
Writing $z=e^{ix}$ as before, it follows that
\begin{equation}\label{sigmaT}
\|\mathcal{V}_r(\sigma(F))(z)\|_{\textup{L}_z^{p'}(\T)}
=\|\mathcal{V}_r(T_{\leq}f) (x)\|_{\textup{L}_x^{p'}([0,2\pi])}.
\end{equation}
Since $1\leq p<\min\{p',r\}$, we are in a position to apply the following result, which is an immediate consequence of \mbox{\cite[Lemma B.1]{OSTTW}.}  

\begin{lemma}\label{B1} 
Under the same assumptions on the exponents $p, r$, there exists a constant $C_{p,r}<\infty$ such that
\begin{equation}\label{Tf}
\|\mathcal{V}_r(T_{\leq}g)\|_{\textup{L}^{p'}(\R)}\leq C_{p,r}\|T\|_{p, p'}\|g\|_{\textup{L}^p(\R)},
\end{equation}
for every $g\in \textup{L}^p(\R).$
\end{lemma}
\noindent Since $\|T\|_{p, p'}\leq (2\pi)^{\frac1{p'}}$ and $\|f\|_{\textup{L}^p(\R)}=\|F\|_{\ell^p(\Z)}$, estimates \eqref{sigmaT} and \eqref{Tf} with $g=f$ imply inequality \eqref{vMPZineq}, concluding the proof of the proposition.
\end{proof}

\noindent Lemma \ref{B1} is proved in \cite{OSTTW} via a bootstrap argument which uses the linear Christ--Kiselev lemma as a black box.
We now present an alternative approach which follows the original argument used to establish \cite[Theorem 1.1]{CK2}. 

\begin{proof}[Alternative proof of Lemma \ref{B1}]
Without loss of generality, we may assume that $\|g\|_{\textup{L}^p(\R)}=1$ and $r<p'$.
For $t\in\R$, 
define the function 
$$\varphi(t)=\int_{-\infty}^t |g(x)|^p \d x.$$
Since $\varphi$ is continuous, non-decreasing, and runs from 0 to 1, we may define points $\{t_k^n\}$ indexed by  $n\in\N$ and $1\leq k< 2^n$ via
$$\varphi(t_k^n)=\frac{k}{2^n},$$
where $t_k^n$ is the smallest such solution if $\varphi$ has a flat piece.
In this range of $n,k$, define the intervals $I_k^n=(t_{k-1}^n,t_k^n]$, and further set 
$I_{2^n}^n=\R\setminus \bigcup_{1\leq k< 2^n} I_k^n$. 
One easily checks, for every $n,k$, that
\begin{equation}\label{normfchi2}
\int_{I_k^n} |g(x)|^p \d x=2^{-n}.
\end{equation}
Given $t\in \R$, write $\varphi(t)=\sum_{n=1}^\infty j_n 2^{-n}$ with $j_n\in\{0,1\}$ as  a binary expansion. If more than one such expansion exists, choose one.
Define $k_n(t)=\sum_{\ell=1}^n j_\ell 2^{n-\ell}$. By continuity of the function $\varphi$, we have that
$$g{\bf 1}_{(-\infty,t]}=\sum_{\{n:\; j_n=1\}} g{\bf 1}_{I_{k_n(t)}^n},$$
as functions in $\textup{L}^p$. Proceed similarly for intervals of the form $[s,\infty)$. It follows that, for any reals $N_j<N_{j+1}$, 
$$|T_\leq g(x,N_{j+1})-T_\leq g(x,N_j)|
=|T(g\mathbf{1}_{[N_j,N_{j+1}]})(x)|
\leq 2\sum_{n=1}^\infty \sup_k |T(g \mathbf{1}_{I_k^n})(x)|,$$
where, for each $n$, the supremum in the last summand is taken over all $1\leq k\leq 2^n$ such that  $I_k^n\subset [N_j,N_{j+1}]$. As a consequence,
\begin{multline*}
\Big(\sum_{j=0}^{K-1} |T_\leq g(x,N_{j+1})-T_\leq g(x,N_j)|^r\Big)^{\frac1r}
\leq
2 \Big(\sum_{j=0}^{K-1}\Big(\sum_{n=1}^\infty \sup_k |T(g \mathbf{1}_{I_k^n})(x)|\Big)^r\Big)^{\frac1r}\\
\leq
2 \sum_{n=1}^{\infty}\Big(\sum_{j=0}^{K-1} \sup_k |T(g \mathbf{1}_{I_k^n})(x)|^r\Big)^{\frac1r}
\leq
2 \sum_{n=1}^{\infty}\Big(\sum_{k=1}^{2^n} |T(g \mathbf{1}_{I_k^n})(x)|^r\Big)^{\frac1r},
\end{multline*}
where the last two inequalities are a consequence of Minkowski's inequality and the fact that,
for each $n$,  $\{I_k^n: 1\leq k\leq 2^n\}$ is a partition of $\R$ into disjoint subintervals.
Hence
\begin{multline*}
\|\mathcal{V}_r(T_{\leq}g)(x)\|_{L_x^{p'}(\R)}
\leq
2\Big(\int_\R \Big|\sum_{n=1}^{\infty}\Big(\sum_{k=1}^{2^n} |T(g \mathbf{1}_{I_k^n})(x)|^r\Big)^{\frac 1r}\Big|^{p'}\d x\Big)^{\frac 1{p'}}\\
\leq
2\sum_{n=1}^\infty\Big(\int_\R\Big(\sum_{k=1}^{2^n} |T(g\mathbf{1}_{I_k^n})(x)|^r\Big)^{\frac{p'}r}\d x\Big)^{\frac1{p'}}
\leq
2\sum_{n=1}^\infty\Big(\sum_{k=1}^{2^n}\Big(\int_\R |T(g \mathbf{1}_{I_k^n})(x)|^{p'} \d x\Big)^{\frac r{p'}}\Big)^{\frac 1r},
\end{multline*}
where Minkowski's integral inequality was used to interchange the $\ell^r$ and the $\textup{L}^{p'}$ norms (recall that $r<p'$). Since the operator $T$ is bounded from $\textup{L}^p$ to $\textup{L}^{p'}$, we appeal to \eqref{normfchi2} to further estimate 
$$\|\mathcal{V}_r(T_{\leq}g)(x)\|_{\textup{L}_x^{p'}(\R)}
\leq
2\|T\|_{p,p'}\sum_{n=1}^\infty\Big(\sum_{k=1}^{2^n} 2^{-\frac{nr}p}\Big)^{\frac 1r}.$$
Since $r>p$, the geometric series converges, concluding the proof of the lemma
with $C_{p,r}=2(2^{\frac1p-\frac1r}-1)^{-1}$.
\end{proof}

\section{Proof of the main theorem}\label{sec:corproofs}

Armed with Propositions \ref{comparingvars} and \ref{varMPZ}, it is now an easy matter to finish the proof of the main theorem.
\begin{proof}[Proof of Theorem \ref{VNLHY}]
Let $1\leq p<2$ and $r>p$.
We deal with the case $r<2$ first.
In this case, given a potential  $F$, we start by considering the set 
$$\mathbb{S}=\{z\in\T: \mathcal{V}_r(\gamma[F])(z)\leq 1\},$$ 
defined in \eqref{defS}. 
For any point $z\in\mathbb{S}$, we have that $\mathcal{V}_r(\gamma[F])(z)\leq 1$, and therefore $\mathcal{V}_r(\sigma[F])(z)\lesssim_r 1$ in view of Corollary \ref{revcomparingvars}. 
As observed before, this implies $\|F\|_{\ell^r}\lesssim_r 1$.
From Proposition \ref{comparingvars}, we obtain 
$$\mathcal{V}_r(\gamma[F])(z)\lesssim_r \mathcal{V}_r(\sigma(F))(z)
+ \Big\|\log\Big(\frac{1+|F_n|}{1-|F_n|}\Big)\Big\|_{\ell^r}.$$
From Proposition \ref{varMPZ}, the elementary inequality $x\leq\frac 12\log(\frac{1+x}{1-x})$, and the inclusion $\ell^p\subset \ell^r$, it follows that 
\begin{equation}\label{firstonlyr<2}
\|\mathcal{V}_r(\gamma[F])\|_{\textup{L}^{p'}(\mathbb{S})}
\lesssim_{p,r}\Big\|\log\Big(\frac{1+|F_n|}{1-|F_n|}\Big)\Big\|_{\ell^p(\Z)}.
\end{equation}
We now consider the complementary set $\T\setminus \mathbb{S}$.
Given $z\in\T\setminus \mathbb{S}$, estimate \eqref{gammasigmaVr} implies
 $\mathcal{V}_r(\sigma[F])(z)\gtrsim_r 1$, and therefore
$$\mathcal{V}_r(\gamma[F])(z)\lesssim_r \mathcal{V}^r_r(\sigma(F))(z)+ \Big\|\log\Big(\frac{1+|F_n|}{1-|F_n|}\Big)\Big\|^r_{\ell^r}.$$
This leads to
$$\|\mathcal{V}_r(\gamma[F])\|^{\frac1r}_{\textup{L}^{\frac{p'}r}(\T\setminus\mathbb{S})}
\lesssim_r 
\|\mathcal{V}^r_r(\sigma(F))\|^{\frac1r}_{\textup{L}^{\frac{p'}r}(\T)}+ \Big\|\log\Big(\frac{1+|F_n|}{1-|F_n|}\Big)\Big\|_{\ell^r(\Z)}.$$
Another application of Proposition \ref{varMPZ} shows that 
$$\|\mathcal{V}^r_r(\sigma(F))\|^{\frac1r}_{\textup{L}^{\frac{p'}r}(\T)}=\|\mathcal{V}_r(\sigma(F))\|_{\textup{L}^{p'}(\T)}\lesssim_{p,r} \|F\|_{\ell^p(\Z)},$$
and so, reasoning as before,
\begin{equation}\label{onlyr<2}
\|\mathcal{V}_r(\gamma[F])\|^{\frac1r}_{\textup{L}^{\frac{p'}r}(\T\setminus\mathbb{S})}
\lesssim_{p,r}\Big\|\log\Big(\frac{1+|F_n|}{1-|F_n|}\Big)\Big\|_{\ell^p(\Z)}.
\end{equation}
Estimates \eqref{firstonlyr<2} and \eqref{onlyr<2} yield the desired inequality \eqref{mainineq1}, as long as $r<2$. 

In order to extend it to values $r\geq 2$, we proceed as follows. 
Fix $1\leq p<2$ and $r\geq 2$.
Set $s=\frac{p+2}2$ and $\mathbb{S}=\{z\in\T: \mathcal{V}_s(\gamma[F])(z)\leq 1\}$. 
Since $s<2$, the preceding discussion implies
$$\|\mathcal{V}_s(\gamma[F])\|_{\textup{L}^{p'}(\mathbb{S})}
+\|\mathcal{V}^{\frac 1s}_s(\gamma[F])\|_{\textup{L}^{p'}(\T\setminus\mathbb{S})}
\lesssim_{p}\Big\|\log\Big(\frac{1+|F_n|}{1-|F_n|}\Big)\Big\|_{\ell^p(\Z)}.$$
If $z\in\mathbb{S}$, then it suffices to recall that the $r$-variation is decreasing in the exponent $r$:
$$\mathcal{V}_r(\gamma[F])(z)\leq \mathcal{V}_s(\gamma[F])(z)$$
If $z\in\T\setminus\mathbb{S}$, then we may estimate:
$$
\mathcal{V}^{\frac1r}_r(\gamma[F])(z)
\leq \mathcal{V}^{\frac1r}_s(\gamma[F])(z)
\leq \mathcal{V}^{\frac1s}_s(\gamma[F])(z),$$ 
where the first inequality follows from the fact that the $r$-variation is decreasing in the exponent $r$, and
the second inequality follows from the fact that $\mathcal{V}_s(\gamma[F])(z)> 1$.

As a consequence, inequality \eqref{mainineq1} holds for arbitrary potentials,
in the full range $1\leq p<2$ and $r>p$.
Since the implicit constant depend only on the exponents $p$ and $r$, the result for general $\ell^p$ sequences follows.
The proof of the theorem is complete.
\end{proof}

\section*{Acknowledgments}
The author would like to thank Christoph Thiele for suggesting the problem and for several stimulating discussions. He is also indebted to Vjeko Kova\v{c}, Jo\~{a}o Pedro Ramos and Gena Uraltsev for various comments and suggestions, and to Jim Wright for an enlightening conversation and a copy of an unpublished joint work with Sandy Davie.

\end{document}